\documentclass[12pt]{article} 

\usepackage[usenames,dvipsnames,condensed]{xcolor}

\usepackage{amsmath}
\usepackage{amsthm}
\usepackage{amsfonts}
\usepackage[english]{babel}
\usepackage[usenames]{xcolor}
\usepackage{graphicx}
\usepackage{soul}
\usepackage{float}
\usepackage{stfloats}
\usepackage{morefloats}
\usepackage{cite}
\usepackage{lscape}
\usepackage{epstopdf}
\usepackage{bbm}
\usepackage[lite]{amsrefs}
\usepackage[OT2,T1]{fontenc}
\DeclareSymbolFont{cyrletters}{OT2}{wncyr}{m}{n}

\usepackage{authblk}

\newcommand{\email}[1]{%
    \normalsize\href{mailto:#1}{\color{black}{#1}}}
    
\usepackage{enumitem}  
\usepackage{calc}
\setlist{labelindent=1pt,itemsep=.5em}
\setlist[itemize]{leftmargin=1.2cm}
\setlist[enumerate]{itemindent=0em,leftmargin=1.2cm}
\setlist[enumerate,1]{label={\upshape(\roman*)}}
    
\usepackage{hyperref} 

\usepackage{marginnote}

\usepackage{comment}

\usepackage{appendix}
\usepackage[lite]{amsrefs}

\renewcommand{\PrintDOI}[1]{\href{http://dx.doi.org/\detokenize{#1}}{doi: \detokenize{#1}}%
  \IfEmptyBibField{pages}{, (to appear in print)}{}}

\allowdisplaybreaks

\newtheorem{theorem}{Theorem}[section]
\newtheorem{corollary}[theorem]{Corollary}

\newtheorem{proposition}[theorem]{Proposition}
\newtheorem{definition}[theorem]{Definition}

\newtheorem{example}[theorem]{Example}

\newtheorem{remark}[theorem]{Remark}

\newtheorem{thm}{Theorem}[section]

\newtheorem{lem}[thm]{Lemma}

\theoremstyle{definition}

\theoremstyle{remark}

\allowdisplaybreaks

\newcommand\rt{\triangleright}

\newcommand{\Z}{\mathbb{Z}}

\title{The derivation problem for quandle algebras}

\author[1]{M. Elhamdadi}
\author[2]{A. Makhlouf} 
\author[3]{S. Silvestrov}
\author[4]{E. Zappala}
\affil[1]{\Affilfont Department of Mathematics, 
University of South Florida, \authorcr \Affilfont 
Tampa, FL 33620, U.S.A. \authorcr \Affilfont
\email{emohamed@math.usf.edu}}
\affil[2]{\Affilfont IRIMAS - d\'epartement de Math\'ematiques,  \authorcr \Affilfont 
Universit\'e de Haute Alsace, \authorcr \Affilfont
6 rue des Fr\`{e}res Lumi\`{e}re, 68093 Mulhouse, France \authorcr \Affilfont
\email{Abdenacer.Makhlouf@uha.fr}}
\affil[3]{\Affilfont Division of Mathematics and Physics,
\authorcr \Affilfont School of Education, Culture and Communication,
\authorcr \Affilfont M\"{a}lardalen University, Box 883, 72123 V{\"a}ster{\aa}s, Sweden \authorcr \Affilfont
\email{sergei.silvestrov@mdh.se}}
\affil[4]{\Affilfont Institute of Mathematics and Statistics, 
\authorcr \Affilfont University of Tartu
\authorcr \Affilfont Narva mnt 18, 51009 Tartu, Estonia
\authorcr \Affilfont
\email{emanuele.amedeo.zappala@ut.ee}, \email{zae@usf.edu}}

\begin{document}

\maketitle

\begin{abstract}
The purpose of this paper is to  introduce and investigate the notion of derivation for quandle algebras.  More precisely, we describe the symmetries on structure constants providing  a characterization  for a linear map to  be a derivation. We  obtain a complete characterization of derivations in the case of  quandle algebras of \emph{dihedral quandles} over fields of characteristic zero, and provide the dimensionality of the Lie algebra of derivations.  Many explicit examples and computations are given over both zero and  positive characteristic. 
	 Furthermore, we investigate inner derivations, in the sense of Schafer for non-associative structures. We obtain necessary conditions for the Lie transformation algebra of quandle algebras of  Alexander quandles, with explicit computations in low dimensions.
\end{abstract}

\tableofcontents

\section{Introduction}
Quandles were introduced independently by Joyce and Matveev  \cites{Joyce,Matveev} in the 1980s as non-associative algebraic structures with the purpose of constructing invariants of knots and links.  Since then, quandles have been extensively studied by both topologists and algebraists.  For more details on quandles, we refer the reader to \cites{EN,Joyce,Matveev}.   Since quandles are set-theoretic objects, it is natural to consider their \emph{linearization} in a similar way to linearization of groups: the group algebras \cite{Passman}.  Quandle rings, which constitute the starting point of the present article, were introduced in \cite{BPS} in an analogous way  to the theory of group rings.  A blending of quandle theory with ring theory was given in \cite{EFT} where various properties of quandle rings were established. Derivations of associative algebras have been studied at least for more than eighty years since the work of Nathan Jacobson "Abstract Derivations and Lie Algebras" in 1937.  Derivations are very closely related to Hochschild cohomology. For an algebra $A$, and an $A$-module $M$, a derivation of $A$ with values in $M$ is a linear map $D:A \rightarrow M$ satisfying $D(xy)= D(x)y +xD(y)$ for all $x,y$ in $A$.  For a fixed element $u \in A$, the derivation $ad(u)(v):=[u,v]=uv-vu,$ for all $v$ in $A$ is called \emph{inner} derivation.  It is well known that $ad(u)$ commute with the Hochschild boundary, thus leading to the fact that derivations of $A$ modulo \emph{inner} derivations (called sometimes \emph{outer} derivations) correspond to the first cohomology group of $A$. Derivations of group algebras have been studied extensively, for the cases of locally compact groups, von Neumann group algebras and discrete groups \cites{Los,Kad,AM,Ghah}.  The theory of finite-dimensional non-associative algebras was given by Schafer in \cite{Sha2} where modules and representations for classes of non-associative algebras were given.  Furthermore, inner derivations of non-associative algebras were studied by Schafer in \cite{Scha} where a definition of inner derivation in the non-associative context is given and it is shown that the derivations of semi-simple alternative 
and Jordan algebras are all inner.

In this article, we  deal with derivations over quandle algebras of some quandles.   We provide a necessary and sufficient condition, on structure constants, for a linear map to be a derivation on any quandle algebra (Proposition \ref{lem:derivation}).
This then allows us to obtain the complete characterization of derivations of the quandle algebra over dihedral quandles (cf.  Theorem~\ref{pro:dihedralderivations}): 

Let $X = \Z_n$ be a dihedral  quandle and $\mathcal{A} = \mathbbm k [X]$ be its quandle algebra, where $char\ \mathbbm k = 0$, then
\begin{enumerate}[label=\rm \arabic*)]
\item If $n$ is odd, the derivation algebra $\mathfrak{D}(\mathcal{A})$ is trivial;
\item If $n = 4k$ for some $k$, the coefficients $c_i^j$ of arbitrary matrices in $\mathfrak{D}(\mathcal{A})$ are determined by the symmetries $c_{t+2d}^x = c_t^{2t+2d-x}$, $c_t^x = -c_t^{x+2k}$ and $c_t^x = c_{t+2k}^{x+2k}$, for all $x,t,d$.
\item 
If $n = 2k$ for some odd $k$, the coefficients $c_i^j$ of matrices in the derivation algebra $\mathfrak{D}(\mathcal{A})$ are characterized by the symmetries $c_{t+2d}^x = c_t^{2t+2d-x}$, and $c_t^x = -c_t^{x+k}$.
\end{enumerate}
As consequences of Theorem~\ref{pro:dihedralderivations}, we obtain the forms of the matrices of derivations in the context of dihedral quandles.  Furthermore, the dimensions of the derivation algebras of the quandle algebra over dihedral quandles are given in Theorem~\ref{dimensions}. Next, we consider the problem of inner derivations in quandle algebras, following previous work of Schafer \cite{Scha}, by means of the notion of Lie transformation algebra. We find some necessary conditions for linear maps on quandle algebras of Alexander quandles to be in the Lie transformation algebra, and perform computations for low degrees, explicitly exhibiting the Lie transformation algebra. We mention that in \cite{NSS} the notion of ``derivation of a quandle'' was introduced, where derivations were defined as twisted analogues of quandle homomorphisms. It is not yet clear how the theory of derivations of quandle algebras and that of derivations of quandles are related.

The article is organized as follows.  In Section~\ref{review} we review the basics of quandles, quandle algebras and give few examples.  Section~\ref{Der} includes the \emph{main results} of the article, in particular the complete characterization of derivations of the quandle algebra over \emph{dihedral quandles} and the dimensions of their derivation algebras.  In Section~\ref{Computations} we give many examples stating the matrix forms of the derivations of quandle algebras over dihedral quandle of cardinality up to $6$ in positive characteristic cases.  This section also contains the derivations of the quandle algebra of all quandles of order $3$ and $4$ in both zero and positive characteristics.  We also give an 
example showing that derivations of the quandle algebra over the conjugation quandle of the symmetric group $S_3$ is trivial.  Section~\ref{LieTransformation} introduces the Lie transformation algebra, that is the smallest Lie algebra containing right and left multiplications of the quandle algebra.  For the case of Alexander quandles a necessary condition on the form of the elements of the Lie transformation algebra is described.  The section ends with  explicit computations of the Lie transformation algebra for quandles of order three.  
 
\section{Quandles and Quandle Algebras}\label{review}
We start by recalling the definition of a quandle and give a few examples.
\begin{definition}A \textit{quandle} is a set $X$ with an operation $\rhd:X \times X \rightarrow X$ such that the following three conditions are satisfied. 
\begin{eqnarray*}
& &\mbox{\rm (I) \ }   \mbox{\rm  For any $x \in X$,
$x \rt x =x$.} \label{axiom1} \\
& & \mbox{\rm (II) \ }\mbox{\rm For any $y,z \in X$, there is a unique $x \in X$ such that 
$ x \rt y=z$.} \label{axiom2} \\
& &\mbox{\rm (III) \ }  
\mbox{\rm For any $x,y,z \in X$, we have
$ (x \rt y) \rt z=(x \rt z)\rt(y \rt z). $} \label{axiom3} 
\end{eqnarray*}
\end{definition}

Here are a few typical examples of quandles.
\begin{enumerate}[label=\rm \arabic*)]
\item
A quandle $X$ is \emph{trivial} if for all $x, y \in X,$ we have $x \rt y=x.$
\item
Let $n$ be a positive integer.  Then $x\rt y=-x+2y$ gives a quandle structure on the set $\mathbb{Z}_n$ of integers modulo $n$, called \emph{dihedral} quandle.
\item
Any $\mathbb{Z}[T, T^{-1}]$-module $M$
is a quandle with
$x \rt y= Tx+(1-T)y$, $x,y \in M$, called an {\it  Alexander  quandle}.
\item
Any group $X=G$ with conjugation $x\rt y=y^{-1} xy$ becomes a quandle.
\end{enumerate}
  
  Let $X$ be a quandle.
  The {\it right multiplication}  $\mathcal{ R}_x:X \rightarrow  X$, by $x \in X$, is defined
by $\mathcal{ R}_x(y) = y \rt x$ for all $y \in X$. Similarly the {\it left multiplication} $\mathcal{ L}_x$
is defined by $\mathcal{ L}_x(y) = x \rt y$. Then $\mathcal{ R}_x$ is a bijection of $X$ by Axiom (II).
The subgroup of the group of bijections of $X$ generated by $\mathcal{ R}_x$, $x \in X$, is 
called the {\it inner automorphism group} of $X$,  and is 
denoted by ${\rm Inn}(X)$. 
A quandle is {\it connected} if ${\rm Inn}(X)$ acts transitively on $X$.  It is called \emph{involutive} if $R_x=R_x^{-1}$ for all $x \in X$ and \emph{latin} if $L_x$ are bijections for all $x \in X$.

For a quandle $(X, \rt)$ and a field $\mathbbm{k}$, we consider, in general, a \textit{nonassociative} algebra $\mathbbm {k}[X]$.   Precisely,  we let $\mathbbm {k}[X]$ be the set of elements that are uniquely expressible in the form $\sum_{x \in X }  a_x e_x$, where $a_x=0$ for almost all $x$, (that is $e_x$ represent the basis elements of $\mathbbm {k}[X]$).  Addition on $\mathbbm {k}[X]$ is defined as usual and the multiplication is given by the following operation, where $x, y \in X$ and $a_x, a_y \in \mathbbm {k}$,
 $$   ( \sum_{x \in X }  a_x e_x) \cdot ( \sum_{ y \in X }  b_y e_y )
 =   \sum_{x, y \in X } a_x b_y e_{x \rt y}. $$
 
 We observe that, in general, the algebra just described \emph{does not} satisfy the self-distributivity condition that characterizes quandles. This happens only in very special cases. To date, we are not aware of any polynomial identity that quandle algebras satisfy in general. 
 
 By analogy with group algebras, the augmentation ideal is defined to be the kernel of the surjective algebra homomorphism $\epsilon: \mathbbm{k}[X] \rightarrow \mathbbm{k}$ such that $\epsilon (\sum_{x \in X }  a_x e_x)=\sum_{x \in X }  a_x $.  It will be denoted by $I_X$.  It is two-sided ideal.  By fixing $x_0 \in X$, one sees that the elements $e_x-e_{x_0}, (x \in X, x \neq x_0)$ form a basis of $I_X$.  We also have the isomorphism of rings $\mathbbm{k}[X]/I_X \cong \mathbbm{k}$.
 
 {\bf Notation:}   Throughout the whole article, it is important to make a distinction between the product in the quandle, denoted by $\rt$, and the product, denoted by $\cdot$, in the quandle algebra.  
 
 As a consequence of the right-distributivity in a quandle $(X, \rt)$, we have the following lemma.
\begin{lem}
In the quandle algebra $\mathbbm {k}[X]$, the set $J_X:=\left\langle e_{x \rt y} -e_{y \rt x}, \; x,y \in X \right\rangle $ generated by the elements of the form $e_{x \rt y} -e_{y \rt x}$ is a \emph{right} ideal of $\mathbbm {k}[X]$.  Furthermore, if $X$ is a \emph{medial} quandle, then $J_X$ is \emph{left} ideal and thus \emph{two sided} ideal. 
\end{lem}
\begin{proof}
	Let $J_X=\left\langle e_{x \rt y} -e_{y \rt x}, \; x,y \in X \right\rangle $, one has
	\[
	(e_{x \rt y} -e_{y \rt x})\cdot e_z= e_{(x \rt y) \rt z} - e_{(y \rt x) \rt z}=e_{(x \rt z)\rt(y \rt z)}-e_{(y \rt z)\rt (x \rt z)}.
	\]
	Thus, by linearity, $J_X$ is a right ideal of $\mathbbm {k}[X]$.
	Now if $X$ is a \emph{medial} quandle, then 
	\[
	 e_z \cdot (e_{x \rt y} -e_{y \rt x})= e_{z \rt z} \cdot (e_{x \rt y} -e_{y \rt x})= e_{(z \rt x) \rt (z \rt y)}- e_{(z \rt y) \rt (z \rt x)}.  
\qedhere 	\]
\end{proof}
\begin{remark}
It is clear that for any quandle $X$, the ideal $J_X$ embeds in the ideal $I_X$. Now since $x \rt y=y \rt x$ for all $x,y$ in the dihedral quandle $\mathbb{Z}_3=\{0,1,2\}$, then $$\{0\}=J_X \subset I_X=\left\langle e_1-e_0, e_2-e_0 \right\rangle. $$
\end{remark}
   
\section{Derivations of Quandle Algebras}
\label{Der}
Let $A$ be a (possibly non-associative) algebra. A linear map $D: A \longrightarrow A$ satisfying the Leibniz rule, i.e.
	$
	D(x\cdot y) = D(x)\cdot y + x\cdot D(y)
	$ 
	for all $x,y\in A$, is called a {\it derivation} of $A$.
The derivations of $A$ form a Lie algebra denoted by $\mathcal D(A)$, and called {\it derivation algebra} of $A$.

In this section we study derivations on some types of quandles.

\begin{definition}
	{\rm 
A derivation on a quandle algebra $\mathbbm {k}[X]$ is a \emph{linear} map $D:\mathbbm{k}[X] \rightarrow \mathbbm{k}[X]$ that verify the Leibniz identity on the vector basis 
$$
D(e_x \cdot e_y) =D(e_x) \cdot e_y + e_x \cdot D(e_y),
$$
where $e_x \cdot e_y=e_{x \rt y}$ and extended by linearity.
}
\end{definition}
 
\subsection{Characterization of Quandle Algebra Derivations }

First, we derive a characterization of derivations on quandle algebras. This will be used in the rest of the article to study the derivation algebras of some important families of quandles, and to determine the inner and outer derivations of the quandle algebra.

\begin{proposition}\label{lem:derivation}
    Let $\mathbbm k$ be a field of any characteristic, let $X$ denotes a finite quandle and let $A = \mathbbm k[X]$ be its quandle algebra. Let $D$ be a derivation of $A$, then the matrix $ (c^j_i)$ representing $D$ in the basis $e_x$, $x\in X$ satisfies the relations
    $$
    c_{x\rt y}^z = c_x^{\hat z} + \sum_{q\in x^{-1}(z)} c_y^{q}
    $$
    for all $x,y,z\in X$, where $\hat z$ is the unique element of $X$ such that $\hat z \rt y = z$, and $x^{-1}(z)$ is the set (possibly empty) of elements $w\in X$ such that $x \rt w = z$. 
\end{proposition}
\begin{proof}
    Using linearity of $D$, we see that $D$ satisfies the Leibniz rule if and only if for all $x,y\in X$, $D$ satisfies the Leibniz rule on $e_x$ and $e_y$. This means that the equation
    $$
    D(e_x\cdot e_y) = D(e_x)\cdot e_y + e_x\cdot D(e_y),
    $$
    holds for all $x,y$. Using Einstein summation convention, the previous equation can be rewritten as
    $$
    c_{x \rt y}^z e_z = c_x^{z'} e_{z' \rt y} + c_y^{z''} e_{x \rt z''}.
    $$
    To conclude we need to equate terms on the right hand side corresponding to $e_z$, for any given $z\in X$. Therefore we take $z' = \hat z$, and we consider all the $z''\in X$ such that $x \rt z'' = z$. 
\end{proof}
\begin{remark}
Observe that $\hat z=z \rt^{-1} y$, and thus when $X$ is an involutive quandle one has $\hat z = z \rt y$, while when the quandle $X$ is latin, i.e. the left multiplication map is also invertible, we see that $x^{-1}(z)$ consists of one single element. Clearly, when $X$ is not connected it may well happen that  $x^{-1}(z) = \emptyset$.
\end{remark}

First we consider derivations of 
quandle algebras of trivial quandles.

\begin{proposition}
 	Let  $(X,\rt)$ be the trivial quandle, then a linear map $D$ is a derivation on the quandle algebra $\mathbbm {k}[X]$ if and only if $Im(D)\subseteq I_X$, that is, $D(e_x)= \sum_{u \in X }  c_x^u e_u$, where $ \sum_{u \in X }  c_x^u =0$ for all $x \in X$.
\end{proposition}

\begin{proof}
	For all $x,y \in X$, let $D(e_x)= \sum_{u \in X }  c_x^u e_u$ and $D(e_y)= \sum_{v \in X }  c_y^v e_v$, then $D$ is a derivation if and only if 
	$$
	\sum_{u \in X }  c_x^u e_u= \sum_{u \in X }  c_x^u e_u + (\sum_{v \in X }  c_y^v)e_x.
	$$
	One then obtains that for any fixed $y \in X,$ $\sum_{v \in X }  c_y^v=0$, that is $D(e_y)\in I_X$ for all $y \in X$. The result follows.
	
\end{proof}

The following example shows that there are infinitely many quandle algebras with nontrivial derivations. 

\begin{example}\label{ex:derivationconjugation}
{\rm 
	Let $G$ be a group with nontrivial center $Z(G) \neq {1}$. Let $x\neq 1$ denote a central element. Consider the quandle algebra corresponding to $G$ endowed with conjugation quandle operation, over a field $\mathbbm k$. Then the following map is a derivation of the quandle algebra
	$$
	D_x(e_y) := e_y - e_{yx},
	$$
	and extended by linearity, where juxtaposition of letters denotes multiplication in $G$ (not conjugation) and $\mathbbm k [G]$ is generated by symbols $e_z$ with $z\in G$, as before. In fact, since for all $x,y\in Z(G)$ we have
	$$
	D_y\circ D_x = D_x + D_y - D_{xy},
	$$
	it follows that the derivations of type $D_x$ for $x\in Z(G)$ span an algebra which is associative and commutative, where multiplication is defined by composition of maps. This is a trivial sub-Lie algebra of the Lie algebra of derivations.
	}
\end{example}

Next, we show that quandle algebra derivations of conjugation latin quandles are restrained to have image inside the augmentation ideal, similarly to the case of trivial quandles. 

\begin{proposition}\label{pro:conjugationaugmentation}
    Let $G$ denote a group with conjugation quandle operation $x\rt y = y^{-1}xy$. Suppose that $G$ is a latin quandle. Let $A:= \mathbbm k[G]$ denote the quandle algebra of $G$. Then, if $D$ is a derivation of $A$, we have that $D(A) \subset \mathcal I_G$.
\end{proposition}
\begin{proof}
    First we show that $D(e_1) = 0$, where $1$ is the unit of $G$, and $e_1$ is the corresponding generator of $A$. Since $e_1$ is a right unit for the algebra $A$, the Leibniz rule gives that $D(e_x) = D(e_x\cdot e_1) = D(e_x)\cdot e_1 + e_x\cdot D(e_1)$, which implies that $e_x\cdot D(e_1) = 0$ for all $x\in G$. Since $G$ is connected, there exists a $g\in G$ such that the conjugation action of $G$ on $g$ gives the whole group $G$. This implies that, writing $D(e_1) = \sum_u c_1^ue_u$, $e_g\cdot D(e_1) = \sum_u c_1^u e_{u^{-1}gu}$ where the terms $e_{u^{-1}gu}$ are all different from each other. It follows that $D(e_1) = 0$. Since $e_1\cdot e_x = e_1$ for all $x\in G$, we have that $D(e_1) = D(e_1\cdot e_x )= D(e_1)\cdot e_x + e_1\cdot D(e_x)$ from which $e_1\cdot D(e_x) = 0$. In order for this to happen, we need that $D(e_x)$ is in the augmentation ideal of $A$, which concludes the proof.
\end{proof}

\begin{remark}
    {\rm 
    In fact, the previous argument to show that $D(e_1) = 0$ can be applied, under the same conditions, to any element $x\in Z(G)$, the center of the group $G$.
    }
\end{remark}

\begin{remark}
    {\rm 
    Computations in Section~\ref{Computations} show that the statement of Proposition~\ref{pro:conjugationaugmentation} in general does not hold when $G$ is not connected. Indeed, also $D(e_1) = 0$ needs not be true.
    }
\end{remark}

\subsection{Dihedral Quandle Algebra Derivations}

The following theorem gives a complete characterization of derivations of the quandle algebra corresponding to the dihedral quandle of  arbitrary cardinality. 
\begin{theorem}\label{pro:dihedralderivations}
Let $X = \Z_n$ be a dihedral  quandle and let $\mathcal{A} = \mathbbm k [X]$ be its quandle algebra, where $char\ \mathbbm k = 0$. Then,
\begin{enumerate}[label=\rm \arabic*)]
    \item If $n$ is odd, the derivation algebra $\mathfrak{D}(\mathcal{A})$ is trivial;
    \item If $n = 4k$ for some $k$, the coefficients $c_i^j$ of arbitrary matrices in $\mathfrak{D}(\mathcal{A})$ are determined by the symmetries $c_{t+2d}^x = c_t^{2t+2d-x}$, $c_t^x = -c_t^{x+2k}$ and $c_t^x = c_{t+2k}^{x+2k}$, for all $x,t,d$.
\item 
If $n = 2k$ for some odd $k$, the coefficients $c_i^j$ of matrices in the derivation algebra $\mathfrak{D}(\mathcal{A})$ are characterized by the symmetries $c_{t+2d}^x = c_t^{2t+2d-x}$, and $c_t^x = -c_t^{x+k}$.
\end{enumerate}
In particular, the symmetries imply that $c_t^{k+t} = 0$ for all $t$. 
\end{theorem}
\begin{proof}
    Let us start with the first case, and let us take a derivation $D$.
    Since the order of $X$ is odd, $2$ is invertible modulo $n$. The characterization of derivations given in Proposition~\ref{lem:derivation} becomes
    $$
        c^m_{2j-i} = c^{2j-m}_i + c^{\frac{m+i}{2}}_j,
    $$
    for all $m,i,j\in X$. If we set $m\rt j = i+p$ for some $p$, we obtain 
    \begin{eqnarray}
    c^{m}_{m+p} &=& c^{2j-m}_{m\rt j-p} + c^{j-\frac{1}{2}p}_j,\label{eqn:m+p}
    \end{eqnarray}   
    while setting $m\rt j = i-p$,  we obtain the equation 
    \begin{eqnarray}
    c^{m}_{m-p} = c^{2j-m}_{m\rt j+p} + c^{j+\frac{1}{2}p}_j.\label{eqn:m-p}
    \end{eqnarray}
    Now, the term $ c^{2j-m}_{m\rt j+p}$ in \eqref{eqn:m-p} can be rewritten, by means of \eqref{eqn:m+p} with $m\rt j$ substituted for $m$, as
    $$
    c^{2j-m}_{m\rt j+p} = c^{j-\frac{1}{2}p}_j + c^m_{m-p}.
    $$
    Substituting the latter in \eqref{eqn:m-p}, we obtain that $c^{j-\frac{1}{2}p}_j + c_j^{j+\frac{1}{2}p} = 0$. 
    
    The latter equation, can be rewritten as 
    \begin{eqnarray}\label{eqn:plusminusqup}
    c_j^{j-q} + c_j^{j+q} &=& 0,
    \end{eqnarray}
    
    and clearly implies also the equation 
    
    \begin{eqnarray}\label{eqn:plusminusqdown}
    c^j_{j-q} + c^j_{j+q} &=& 0,
    \end{eqnarray}
    since with $j' = j- q$ we have $c_{j+q}^j = c_{j'}^{j'-q} = - c_{j'}^{j'+q} = - c_{j-q}^j$.
    
    Let us rewrite \eqref{eqn:m+p} with $2p$ instead of $p$, to get 
    \begin{eqnarray}\label{eqn:m+p-p<->2p}
    c_{m+2p}^m &=& c_j^{j-p}+ c_{2j-m-2p}^{2j-m},
    \end{eqnarray}
    while if we rewrite \eqref{eqn:m+p} with $m-p$ in place if $m$, we obtain 
     \begin{eqnarray}\label{eqn:m+p-m<->m-p}
    c_m^{m-p} &=& c_j^{j-\frac{1}{2}p} + c_{2j-m}^{2j-m+p}.
    \end{eqnarray}
     Consequently, the term $c_j^{j-p}$ in \eqref{eqn:m+p-p<->2p} can be rewritten by means of \eqref{eqn:m+p-m<->m-p} as $c_j^{j-p} = c_k^{k-\frac{1}{2}p} + c_{2k-j}^{2k-j+p}$, and \eqref{eqn:m+p-p<->2p} becomes
    \begin{eqnarray}\label{eqn:m+p-p<->2p'}
    c_{m+2p}^m &=& c_k^{k-\frac{1}{2}p} + c_{2k-j}^{2k-j+p} + c_{2j-m-2p}^{2j-m}.
    \end{eqnarray}
    Taking $j=m$ and using \eqref{eqn:plusminusqdown} to rewrite $c_{m-2p}^m$ gives the equation
    \begin{eqnarray}\label{eqn:2left}
    2c_{m+2p}^m &=& c_k^{k-\frac{1}{2}p} + c_{2k-m}^{2k-m+p}.
    \end{eqnarray}
    
    Now, if we apply \eqref{eqn:m+p-m<->m-p} to the RHS of \eqref{eqn:2left}, where we use $j$ instead of $k$, we get
    \begin{eqnarray}\label{eqn:2leftfinal}
    2c_{m+2p}^m &=& c_m^{m-p}.
    \end{eqnarray}
    We observe that \eqref{eqn:2leftfinal} can be equivalently rewritten also as $2c_{m+p}^m = c_m^{m-\frac{1}{2}p}$ and $2c_{m-p}^m = c_m^{m+\frac{1}{2}p}$. We will refer to the latter forms of \eqref{eqn:2leftfinal} as \eqref{eqn:2leftfinal}, without further specification. Let us now consider \eqref{eqn:2left} with $j = k$, where we rewrite the rightmost term by means of \eqref{eqn:2leftfinal}, and derive the equality $2c_{m+2p}^m = c_j^{j-\frac{1}{2}p}+2c_{2j-m-2p}^{2j-m}$. From the latter follows
    \begin{eqnarray}\label{eqn:4left}
    4c_{m+2p}^m &=& c_m^{m-\frac{1}{2}p},
    \end{eqnarray}
    having used $m=j$ and \eqref{eqn:plusminusqdown}. Now, we obtain $4c_{m+2p}^m = 2c_{m+p}^m$ and, therefore, the equation 
    \begin{eqnarray}\label{eqn:1/2}
    2c_{m+2p}^m = c_{m+p}^m,
    \end{eqnarray}
    which is readily rewritten also as 
    \begin{eqnarray}\label{eqn:updown}
    c_m^{m-p} &=& c^m_{m+p}.   
    \end{eqnarray}

    Observe that \eqref{eqn:updown} along with \eqref{eqn:plusminusqdown} gives also the symmetry
    \begin{eqnarray}\label{eqn:updown'}
    c_{m+p}^m + c_m^{m+p} &=& 0,
    \end{eqnarray}
    and, moreover, \eqref{eqn:1/2} can be also seen to imply 
    \begin{eqnarray}\label{eqn:1/2superscript}
    2c_m^{m+2p} &=& c_m^{m+p}.
    \end{eqnarray}
    
    Let us now consider \eqref{eqn:m+p-p<->2p}, and let us take $m+2p$ in place of $j$, to get the equation 
    \begin{eqnarray}\label{eqn:m+p-p<->2pandj<->m+2p}
    c_{m+2p}^m &=& c_{m+2p}^{m+p} + c_{m+2p}^{m+4p}.
    \end{eqnarray}
    By means of \eqref{eqn:1/2superscript} applied twice, we have that $c_{m+2p}^{m+4p} = \frac{1}{4}c_{m+2p}^{m+p}$. So, \eqref{eqn:m+p-p<->2pandj<->m+2p} becomes $c_{m+2p}^m = \frac{5}{4}c_{m+2p}^{m+p} = \frac{5}{4}c_{m+p}^m$, where the last equality is a consequence of \eqref{eqn:updown}. But \eqref{eqn:1/2} gives that $4c_{m+2p}^m = 2c_{m+p}^m$, so we find that $3c_{m+p}^m = 0$, from which $c_k^t = 0$ for all $k,t$, implying that $D$ is the trivial derivation. 
    
    Let us now suppose that $n = 4k$ for some $k$. We distinguish the cases when the derivation $D$ is evaluated on a basis vector $e_t$ of $A$ with even and odd $t\in X$. Let us set, as before, $D(e_x) = \sum_\ell c_x^\ell e_\ell$. 
   
   We apply Leibniz rule to a product of arbitrary basis vectors $e_t\cdot e_{t+d}$. Then we have, by definition, the equality 
   $$\sum_x c_{t+2d}^x e_x = \sum_y c_t^y e_{2t +2d-y} + \sum_zc_{t+d}^z e_{2z-t}.$$
   For even $t$, applying Proposition~\ref{lem:derivation} 
   we see that there are two subcases, corresponding to the parity of $x$, since the set $z^{-1}(x)$ is either empty, or has two elements. Specifically, when $x$ is odd, there is no $z$ such that $2z-t = x$, so that we have an equality
   \begin{eqnarray}
   c_{t+2d}^x &=& c_t^{2t+2d-x}, \label{eqn:evenodd}
   \end{eqnarray}
   for all even $x$, even $t$ and all $d\in \Z_{4k}$. When $x$ is odd there are exactly two values of $z$ that solve the equation $2z-t = x$, namely $\frac{t+x}{2} + 2ik$, $i = 0,1$. The corresponding equation is therefore
   \begin{eqnarray}
    c_{t+2d}^x &=& c_t^{2t+2d-x} + c_{t+d}^{\frac{t+x}{2}} + c_{t+d}^{\frac{t+x}{2}+2k}, \label{eqn:eveneven}
   \end{eqnarray}
   for all even $t, x$ and all $d\in Z_{4k}$. By direct inspection, it is easy to see that when $t$ is odd we obtain the same equations where the roles of $x$ even and odd is exchanged. Explicitly one gets \eqref{eqn:evenodd} when $x$ is even, and \eqref{eqn:eveneven} when $x$ is odd. Observe that $\frac{x+t}{2}$ in \eqref{eqn:eveneven} does not depend on $d$, as already seen in  Proposition~\ref{lem:derivation}, 
   but it only depends on the solutions to the equation $t\rt z = x$ with respect to $z$. We now consider the case $t$ even, as the same discussion can be applied, mutatis mutandis, to the odd case, just by inverting the roles of the parities of $x$. 
   
   First, observe that from \eqref{eqn:evenodd} with $d=0$ we have $c_t^x = c_{t}^{2t-x}$, while substituting $t+2k$ for $t$ and taking $d=k$ we obtain the equation $c_t^x= c_{t+2k}^{2t+2k-x}$. Both equations combined give us 
   \begin{eqnarray}
   c_t^r &=& c_{t+2k}^{r+2k}
   \end{eqnarray}
   for all odd $r$, and even $x$. 
   
   Then, we want to show that \eqref{eqn:evenodd} holds indeed when $x$ is even as well. To do so, let us determine values $x'$, $t'$ and $d'$ such that $c_{t'+2d'}^{x'} = c_t^{2t+2d-x}$ so that we can substitute in \eqref{eqn:eveneven} with $t,x,d$. We choose $t' = t+2d$, $d' = - d$ and $x' = 2t+2d-x$ and substitute $c_{t'+2d'}^{x'}$ determined by \eqref{eqn:eveneven} for $t',x',d'$ into the initial \eqref{eqn:eveneven} with $t,x,d$. We obtain 
   \begin{eqnarray}
   c_t^{2t+2d-x} &=& c_{t+2d}^x + c_{t+d}^{\frac{3t-x}{2}+2d} + c_{t+d}^{\frac{3t-x}{2}+2d+ 2k},
   \end{eqnarray}
   which substituted in \eqref{eqn:eveneven} gives 
   \begin{eqnarray}
   c_{t+d}^{\frac{3t-x}{2}+2d} + c_{t+d}^{\frac{3t-x}{2}+2d+ 2k} + c_{t+d}^{\frac{t+x}{2}} + c_{t+d}^{\frac{t+x}{2}+2k} &=& 0.\label{eqn:subt+d}
   \end{eqnarray}
   
   Let us now consider \eqref{eqn:eveneven} with $d=k$. Namely
   \begin{eqnarray}
   c_{t+2k}^x &=& c_t^{2t+2k-x} + c_{t+k}^{\frac{x+t}{2}} + c_{t+k}^{\frac{x+t}{2}+2k}. \label{eqn:evenevendk}
   \end{eqnarray}
   Replace now $x$ with $2t+2k-x$, which is even as well, to obtain the equation 
   \begin{eqnarray*}
   c_{t+2k}^{2t+2k-x} &=& c_t^x + c_{t+k}^{\frac{3t-x+2k}{2}} + c_{t+k}^{\frac{3t-x+2k}{2}+2k},
   \end{eqnarray*}
   which now we rewrite by taking $t+2k$ instead of $t$ to obtain
   \begin{eqnarray*}
   c_t^{2t+2k-x} &=& c_{t+2k}^x + c_{t+3k}^{\frac{3t-x}{2}} + c_{t+3k}^{\frac{3t-x}{2}+2k}.
   \end{eqnarray*}
   We sum the last equation with \eqref{eqn:evenevendk} to obtain 
   \begin{eqnarray*}
   c_{t+k}^{\frac{x+t}{2}} + c_{t+k}^{\frac{x+t}{2}+2k} + c_{t+3k}^{\frac{3t-x}{2}} + c_{t+3k}^{\frac{3t-x}{2}+2k} &=& 0.
   \end{eqnarray*}
   If we substitute $t+k$ instead of $t$, and $x-k$ instead of $x$ in the preceding equation we obtain
   \begin{eqnarray}
   c_{t+2k}^{\frac{x+t}{2}} + c_{t+2k}^{\frac{x+t}{2}+2k} + c_t^{\frac{3t-x}{2}+2k} + c_t^{\frac{3t-x}{2}} &=& 0. \label{eqn:eveneven+}
   \end{eqnarray}
  From \eqref{eqn:subt+d} (with $d=0$) and \eqref{eqn:eveneven+} combined it follows that 
  \begin{eqnarray*}
  c_t^{\frac{x+t}{2}} + c_t^{\frac{x+t}{2}+2k} + c_{t+2k}^{\frac{x+t}{2}} + c_{t+2k}^{\frac{x+t}{2}+2k}&=& 0,
  \end{eqnarray*}
   which, assuming that $\frac{x+t}{2}$ is odd and applying the symmetry $c_t^r = c_{t+2k}^{r+2k}$ previously proved for odd $r$, gives us $2c_{t+2k}^{\frac{t+x}{2}} + 2c_{t+2k}^{\frac{t+x}{2}+2k} = 0$, whence $c_{t+2k}^{\frac{t+x}{2}} = - c_{t+2k}^{\frac{t+x}{2}+2k}$ since we are in characteristic different from $2$. From \eqref{eqn:eveneven} we obtain that whenever $x+t$ is not divisible by $4$, it holds that $c_{t+2d}^x = c_t^{2t+2d-x}$ for all $d$. In fact, the same reasoning gives that $c_{t+2k}^{\frac{3t-x}{2}} = - c_{t+2k}^{\frac{3t-x}{2}+2k}$. But, all the odd elements in $\Z_{4k}$ can be written as $\frac{x+t}{2}$ or $\frac{3t-x}{2}$ with $x$ and $t$ even. In particular, 
   \begin{eqnarray}
   c_t^r &=& - c_t^{r+2k} \label{eqn:oddopposite}
   \end{eqnarray}
   for all even $t$ and odd $r$.  
   
   We now focus on the case when $\frac{x+t}{2}$ is even, i.e. when either both $x$ and $t$ are divisible by $4$ or neither of them is. Take $d$ odd arbitrary, then it follows that $\frac{t-2d+x}{2}$ is odd and we can replace $t$ with $t' := t-2d$, from which we have that $c_t^x = c_{t'+2d}^x = c_{t'}^{2t'+2d-x}$. So we have obtained $c_t^x = c_{t-2d}^{\frac{2t-2d-x}{2}}$ which, upon changing $t$ with $t+2d$ gives us $c_{t+2d}^x = c_t^{2t+2d-x}$, as required. It follows that we have also proved, using \eqref{eqn:eveneven}, that $c_{t+d}^{\frac{x+t}{2}} = - c_{t+d}^{\frac{x+t}{2}+2k}$, when $d$ is odd and $\frac{x+t}{2}$ is even. Observe that in fact this would follow directly by applying the same proof of the case $\frac{x+t}{2}$ odd with $t$ even to the analogous equation to \eqref{eqn:evenodd} for $t$ is odd. We have decided to show that the cases $t$ even and odd can be handled independently and proved that $c_{t+2d}^x = c_t^{2t+2d-x}$ holds true when $\frac{x+t}{2}$ is even and $d$ is odd without using the case $t$ odd. 
   
   To complete proving our claim, we need to show that $c_{t+2d}^x = c_t^{2t+2d-x}$ holds when $\frac{x+t}{2}$ is even and $d$ is even. We have, from previous cases, that 
   $$
   c_{t+2d}^x = c_{t+2+2(d-1)}^x
   = c_{t+2}^{2(t+2)+2(d-1)-x}
   = c_{t+2}^{2t+2(d+1)-x}.
   $$
   Then,
   $$
   c_{t+2}^{2t+2(d+1)-x} = c_{t+2+2d'}^{x+2d'+2-2d}
   = c_{t+2(d'+1)}^{x+2(d'+1)-2d} 
   = c_t^{2t+2d-x},
   $$
   where in the first equality we have used \eqref{eqn:eveneven} with $x' = 2t +2 -x$, $t' = t+2$ which imply $\frac{x'+t'}{2}$ even and it holds for arbitrary $d'$, the second and third equalities are simply obtained by rebracketing and applying the elementary operations, and the last equality is again \eqref{eqn:eveneven} where we use the fact that $d'$ is arbitrary, and can therefore be chosen odd, which allows us to apply one of the previous cases. This completes the claim that $c_{t+2d}^x = c_t^{2t+2d-x}$ for all even $t$, and all $x$ and $d$. Therefore, by an easy inspection, $c_t^x = -c_t^{x+2k}$, and $c_t^x = c_{t+2k}^{x+2k}$ for all $x$ and all even $t$. 
   
   An identical procedure shows that the same symmetries apply when $t$ is odd, as the equations for this case are identical (but with $x$ exchanged) to the even $t$ case.
   
   Now the result follows. In fact, from the symmetry $c_t^x = c_{t+2k}^{x+2k}$ we see that the matrix representing a derivation $D$ has diagonal blocks equal, and the same holds for anti-diagonal blocks.  From $c_{t}^x = - c_{t}^{x+2k}$ it follows that the top left and top right blocks are the negative of each other. This fact clearly implies the same result for the lower blocks, by virtue of the symmetry $c_t^x = c_{t+2k}^x$. These symmetries characterize the derivations of $\mathbbm k[\Z_{4k}]$ along with the equations $c_{t+2d}^x = c_t^{2t+2d-x}$.
   
   Let us now consider the case $n = 2k$. Proceeding as in the previous case, we see that derivations are characterized by the equations 
   \begin{eqnarray}
   c_{t+2d}^x &=& c_t^{2t+2d-x}, \label{eqn:evenoddk}
   \end{eqnarray}
   for even $t$ and odd $x$, and 
   \begin{eqnarray}
    c_{t+2d}^x &=& c_t^{2t+2d-x} + c_{t+d}^{\frac{t+x}{2}} + c_{t+d}^{\frac{t+x}{2}+k}, \label{eqn:evenevenk}
   \end{eqnarray}
   for even $t$ and even $x$. When $t$ is odd, we obtain identical equations with the parities of $x$ exchanged. Let us consider the case $t$ even. Observe that \eqref{eqn:evenevenk} with $x$ substituted with $2t-x$ (and $d=0$) gives the equation
   \begin{eqnarray}
    c_t^{2t-x} &=& c_t^x + c_t^{\frac{3t-x}{2}} + c_t^{\frac{3t-x}{2}+k} \label{eqn:evenevenk3t}.
   \end{eqnarray}
   When $x$ and $t$ are such that $\frac{x+t}{2}$ is odd, then by \eqref{eqn:evenoddk} applied to $t$ and $\frac{x+t}{2}$ instead of $x$, we get
   $$
    c_t^{\frac{x+t}{2}} = c_t^{2t-\frac{x+t}{2}} = c_t^{\frac{3t-x}{2}}
   $$
   Using this in \eqref{eqn:evenevenk3t} and \eqref{eqn:evenevenk} gives that $2c_t^{\frac{3t-x}{2}} + 2c_t^{\frac{3t-x}{2}+k} = 0$, which implies that $c_t^{\frac{3t-x}{2}} = -c_t^{\frac{3t-x}{2}+k}$. By change of variables one sees that we also have $c_t^{\frac{x+t}{2}} = -c_t^{\frac{x+t}{2}+k}$, and therefore $c_t^x = c_t^{2t-x}$ whenever $x$ and $t$ are even such that $\frac{x+t}{2}$ is odd. An analysis as in the case $4k$ distinguishing the cases based on the parity of $\frac{x+t}{2}$ and $d$ shows that for all even $x, t$, and for all $d$ we have $c_{t+d}^{\frac{x+t}{2}} = - c_t^{\frac{x+t}{2}+k}$ and $c_{t+2d}^x = c_t^{2t+2d-x}$. The case when $t$ is odd is treated similarly, making use of the even $t$ case. 
   \end{proof}
In the following, we describe the derivation matrices of dihedral quandle algebras.
\begin{corollary}\label{cor:periodic4}
Let $X = \Z_n$ be a dihedral quandle and let $\mathcal{A} = \mathbbm k [X]$ be its quandle algebra, where $char\ \mathbbm k = 0.$ 
Then, 
\begin{enumerate}[label=\rm \arabic*)]
    \item 
    If $n = 4k$, then the matrix of any derivation $D$ can be written as 
    $$D= \left( \begin{array}{rr}
 P & -P\\
-P & P
\end{array} \right)$$
for some matrix $P$ of size $2k\times 2k$;
\item 
If $n = 2k$ for some odd $k$, then the matrix of any derivation $D$ can be written as
$$D= \left( \begin{array}{rr}
 R & -R 
\end{array} \right),$$
for some matrix $R$ of size $k\times k$. 
\end{enumerate}
\end{corollary}
\begin{proof}
From Theorem~\ref{pro:dihedralderivations}, case $2$, we have that the matrix of a derivation $D$ satisfies the symmetries 
\begin{eqnarray}
 c_t^x &=& c_{t+2k}^{x+2k} \label{eqn:symmetry14k}\\
 c_t^x &=& -c_t^{x+2k}. \label{eqn:symmetry24k}
\end{eqnarray}
From Equation~(\ref{eqn:symmetry14k}) it follows that $D$ is of the form 
$$D= \left( \begin{array}{rr}
 P & Q\\
Q & P
\end{array} \right),$$
while from Equation~(\ref{eqn:symmetry24k}) we have that $Q = -P$. This completes the proof of $1$. 

From Theorem~\ref{pro:dihedralderivations}, case $2$, we have that Equation~(\ref{eqn:symmetry24k}) holds (although Equation~(\ref{eqn:symmetry14k} does not). Therefore, we can write the matrix of a derivation $D$ as 
$$D= \left( \begin{array}{rr}
 R & -R 
\end{array} \right),$$
for some $R$. 
\end{proof}

\begin{corollary}
With the above notations,  let $n = 4k$ and  $k$  even. We have symmetries $c_t^x = c_{t+k}^{x+k}$. Precisely, the derivations of $A$ can be written as matrices of the form
$$D= \left( \begin{array}{rrrr}
U & V & -U& -V\\
-V &U &V & -U \\ 
-U & -V & U& V\\
V &-U &-V& U 
\end{array} \right),$$
where $U$ and $V$ are $k$ by $k$ matrices.
\end{corollary}
\begin{proof}
     Let $D$ be a derivation of the quandle algebra $A$ when $n = 4k$ for some even $k$. Consider the symmetry 
    \begin{equation}
    c_{t+2d}^x = c_t^{2t+2d-x} \label{eqn:Mastersymmetry}
    \end{equation}
    proved in Theorem~\ref{pro:dihedralderivations}, and take $d = k/2$. It follows that $c^x_{t+k} = c_t^{2t+k-x}$ which, upon changing variable $x$ to $x+k$, gives $ c_{t+k}^{x+k} = c_t^{2t-x}$. By \eqref{eqn:Mastersymmetry} with $d=0$, it follows that $c_t^x = c_{t+k}^{x+k}$ for all $x$ and $k$. From Theorem~\ref{pro:dihedralderivations} we have that $D$ is determined by a square matrix $B$ of order $2k$ and it has the shape $ D = \bigl(\begin{smallmatrix}
P && -P\\
-P && P
\end{smallmatrix}\bigr)$. The latter symmetry along with $c_t^x = c_{t+k}^{x+k}$ shows that $P = \bigl(\begin{smallmatrix}
U && V\\
-V && U
\end{smallmatrix}\bigr)$ for some square matrices $U$ and $V$, of order $k$, which shows that the matrices of the derivation algebra of $A$ 
have the asserted form. 
\end{proof}
\begin{remark}
When $k$ is odd, we have that $c_{t+k}^{x+k} = c_{t+1+2\frac{k-1}{2}}^{x+k} = c_{t+1}^{2t-x+1}$, where we have used \eqref{eqn:Mastersymmetry}. This in turn implies, by another application of \eqref{eqn:Mastersymmetry}, that $c_{t+k}^{x+k} = c_{t+1}^{x+1}$.
Then, we have the following  symmetries $c_t^x = c_{t+k-1}^{x+k-1}$.
\end{remark}

\begin{theorem}\label{dimensions}
    Let notation be as above. Then, the dimension $N$ of the derivation algebra of $A$ is given by $N=2k$ if $k=2l$ and $N=2k-1$ if $k=2l+1$.
\end{theorem}
\begin{proof}
   From \eqref{eqn:Mastersymmetry} above with $d=1$ and $x$ replaced by $x+2$, we get $c_{t+2}^{x+2} = c_t^{2t-x} = c_t^x$ for all $t,x$. This shows that a matrix $D$ corresponding to a derivation of $A$ is determined by $2\times 2$ block matrices from its first two rows. From Corollary~\ref{cor:periodic4}, moreover, we have that we can restrict ourselves to the first half of the two rows, when counting the free parameters that give the dimension of the derivation algebra. Let $k$ be even. Then, \eqref{eqn:Mastersymmetry} where $t=0$ and $d = 0$ gives that $c_0^x = c_0^{-x} = c_0^{4k-x}$ for $x = 1, \ldots , k-1$. But, since we have the symmetry $c_t^x = -c_t^{x+2k}$, it follows that $c_0^x$ with $x = k+2, \ldots, 2k-1$ are obtained from the coefficients $c_0^x$ with $x= 1, \ldots , k$, taken with negative sign and in opposite order. Moreover, the same reasoning shows that $c_0^k = 0$. It follows that from the first row we have $k$ free parameters contributing to the dimension of the derivation algebra. For the second row, i.e. with $k = 1$, we have $c_1^x = c_1^{2-x}$, so we can apply the same procedure as above, but with a shift of $2$. Since $c_1^0 = c_1^2$, again we have a contribution of $k$ free parameters to the dimension of the derivation algebra. When $k$ is odd, one proceeds similarly, where the main difference is that the second row now contributes by a term of $k-1$. This completes the proof. 
\end{proof}

In particular, we have the following explicit matrices for small values of $n$. In the sequel the $a_i$ denote free parameters in $\mathbbm{k}$.

\begin{table}[H]
\caption{Matrices of the Derivations for small even  values of $k$}
\label{Table1}
\begin{center}
\begin{tabular}{ |c|c|c|} 
 \hline
$n=4k $  & U & V \\ 
\hline
$k=2$ & $\left( \begin{array}{rr}
a_1& a_2 \\
a_3 &a_4 
\end{array} \right) $ & $\left( \begin{array}{rr}
0 & -a_2 \\
a_3 &0 
\end{array} \right) $  \\ 
  \hline
  $k=4$ & $\left( \begin{array}{rrrr}
a_1 & a_2 & a_3& a_4\\
a_5 & a_6 & a_5 & a_7 \\
a_3 & a_2 & a_1 & a_2 \\
a_8 & a_7 & a_5  & a_6
\end{array} \right) $ & $\left( \begin{array}{rrrr}
0 & -a_4 & -a_3& -a_2\\
a_8 & 0 & -a_8 & -a_7 \\
a_3 & a_4 & 0 & -a_4 \\
a_5 & a_7 & a_8  & 0
\end{array} \right) $  \\ 
  \hline
  $k=6$ & $\left( \begin{array}{rrrrrr}
a_1 & a_2 & a_3& a_4& a_5& a_6\\
a_7 & a_8 & a_7& a_9& a_{10}& a_{11}\\
a_3 & a_2 & a_1& a_2& a_3& a_4\\
a_{10} & a_9 & a_7& a_8& a_7& a_9\\
a_5 & a_4 & a_3& a_2& a_1& a_2\\
a_{12} & a_{11} & a_{10}& a_9& a_7& a_8
\end{array} \right) $ & $\left( \begin{array}{rrrrrr}
0 & -a_6 &- a_5& -a_4& -a_3& -a_2\\
a_{12} & 0 & -a_{12}& -a_{11}& -a_{10}& -a_{9}\\
a_5 & a_6 & 0& -a_6& -a_5& -a_4\\
a_{10} & a_{11} & a_{12}&0& -a_{12}& -a_{11}\\
a_3 & a_4 & a_5& a_6& 0& -a_6\\
a_{7} & a_{9} & a_{10}& a_{11}& a_{12}& 0
\end{array} \right) $  \\ 
  \hline
\end{tabular}
\end{center}
\end{table}

\section{Computations and Examples }\label{Computations}
In this section, we provide examples of derivations for given quandle algebras,  computed either by hands or   using a computer algebra system : the software Mathematica. 

\subsection{Derivations of  Quandle Algebras of Quandles of order 3 and 4}
Since the numbers of isomorphism classes of quandles of order $3$ and $4$, are respectively $3$ and $7$, we include the results of the derivations over their quandle algebras here.

In Table~\ref{Table2}, the leftmost column gives the Cayley table of each quandle of order $3$.  In each Cayley table the $(i,j)$-entry represents $i \rt j$, for example, in the second row of Table~\ref{Table2} we have $1 \rt 3=2$.  The second column of Table~\ref{Table2} provides the derivation matrices on characteristics zero and the third one in characteristics 3.

\begin{table}[H]
\caption{Matrices of the Derivations for quandles of cardinality  $3$}
\label{Table2}
\begin{center}
\begin{tabular}{ |c|c|c|} 
 \hline
$Order=3 $  & Char 0   & Char  3 \\ 
\hline
\small{
$\left[ \begin{array}{c}
1 \;1 \;1  \\
2 \;2 \;2 \\
3 \;3\; 3 
\end{array} \right] $} & $\left( \begin{array}{ccc}
a_1 & a_2 & a_3 \\
a_4 & a_5 & a_6 \\
-a_1-a_4 & -a_2-a_5 & -a_3-a_6
\end{array} \right) $ 
 & $\left( \begin{array}{ccc}
-a_1-a_2 & -a_1-a_2 & -a_1-a_2 \\
a_2 & a_2 & a_2 \\
a_1 & a_1 & a_1
\end{array} \right) $\\ 
\hline
\small{
$\left[ \begin{array}{c}
1 \;1 \;2  \\
2 \;2 \;1 \\
3 \;3\; 3 
\end{array} \right] $} & $\left( \begin{array}{ccc}
a_1 & -a_1 & 0 \\
-a_1 & a_1 & 0 \\
0 & 0 & 0
\end{array} \right) $ 
 & $\left( \begin{array}{ccc}
a_1 & -a_1 & 0 \\
-a_1 & a_1 & 0 \\
0 & 0 & 0
\end{array} \right) $\\ 
  \hline
\small{
$\left[ \begin{array}{c}
1 \;3 \;2  \\
3 \;2 \;1 \\
2 \;1\; 3 
\end{array} \right] $} & $ \begin{array}{ccc}
\mbox{Zero\ Matrix}
\end{array}  $ 
 & $\left( \begin{array}{ccc}
0 & a_1+a_2 & a_2 \\
a_1 & 0 & -a_2 \\
-a_1 & -a_1-a_2 & 0
\end{array} \right) $\\ 
  \hline
\end{tabular}
\end{center}
\end{table}

The derivations for quandle algebras corresponding to the seven classes of quandles of cardinality 4 are given in Table~\ref{Table3} and Table~\ref{Table4}, where the first one refers to characteristic $0$, and the second one to characteristic $2$. 

\begin{table}[H] 
\caption{
Derivations for quandles of cardinality  $4$ in characteristic 0}
\label{Table3}
\begin{center}
\begin{tabular}{ |c|c|} 
 \hline
$Order=4 $  & Char 0    \\ 
\hline
\small{
$\left[ \begin{array}{c}
1\; 1\; 1\; 1  \\
2\; 2\; 2\; 2 \\
3\; 3\; 3\; 3 \\ 
4\; 4\; 4\; 4
\end{array} \right] $} & $\left( \begin{array}{cccc}
a_1 & a_2 & a_3 & a_4 \\
a_5 & a_6 & a_7& a_8 \\
a_9 & a_{10} & a_{11}&a_{12}\\
-a_1-a_5-a_9 & -a_2-a_6-a_{10} & -a_3-a_7-a_{11}&-a_4-a_8-a_{12}
\end{array} \right) $ 
 \\ 
\hline
\small{
$\left[ \begin{array}{c}
1\; 1\; 1\; 1  \\
2\; 2\; 2\; 3 \\
3\; 3\; 3\; 2 \\ 
4\; 4\; 4\; 4
\end{array} \right] $} 
 & Zero Matrix\\ 
\hline
\small{
$\left[ \begin{array}{c}
1\; 1\; 1\; 2  \\
2\; 2\; 2\; 3 \\
3\; 3\; 3\; 1 \\ 
4\; 4\; 4\; 4
\end{array} \right] $} 
 & $\left( \begin{array}{cccc}
a_1 & a_2 & -a_1-a_2 & 0 \\
-a_1-a_2 & a_1 & a_2 &0 \\
a_2 & -a_1-a_2 & a_1 &0\\
0 & 0 & 0 & 0
\end{array} \right) $\\ 
\hline
\small{
$\left[ \begin{array}{c}
1\; 1\; 1\; 1  \\
2\; 2\; 4\; 3 \\
3\; 4\; 3\; 2 \\ 
4\; 3\; 2\; 4
\end{array} \right] $ }
 & Zero Matrix \\ 
\hline
\small{
$\left[ \begin{array}{c}
1\; 1\; 2\; 2  \\
2\; 2\; 1\; 1 \\
3\; 3\; 3\; 3 \\ 
4\; 4\; 4\; 4
\end{array} \right] $ }
 & $\left( \begin{array}{cccc}
a_1 & -a_1 & 0& 0 \\
-a_1 & a_1 & 0& 0 \\
a_2 & a_2 & a_3&a_4\\
-a_2 & -a_2 & -a_3 &-a_4
\end{array} \right) $\\ 
\hline
\small{
$\left[ \begin{array}{c}
1\; 1\; 2\; 2  \\
2\; 2\; 1\; 1 \\
4\; 4\; 3\; 3 \\ 
3\; 3\; 4\; 4
\end{array} \right] $ } 
 & $\left( \begin{array}{cccc}
a_1& -a_1 & 0 & 0 \\
-a_1 & a_1 & 0 & 0 \\
0 & 0 & a_2& -a_2\\
0 & 0 & -a_2 &a_2
\end{array} \right) $\\ 
\hline
\small{
$\left[ \begin{array}{c}
1\; 4\; 2\; 3  \\
3\; 2\; 4\; 1 \\
4\; 1\; 3\; 2 \\ 
2\; 3\; 1\; 4
\end{array} \right] $ }
& Zero Matrix\\
\hline
\end{tabular}
\end{center}
\end{table}
\begin{table}[H]
\caption{
Derivations for quandles of cardinality  $4$ in characteristic 2}
\label{Table4}
\begin{center}
\begin{tabular}{ |c|c|} 
 \hline
$Order=4 $  &  Char  2 \\ 
\hline
{\small 
$\left[ \begin{array}{c}
1\; 1\; 1\; 1  \\
2\; 2\; 2\; 2 \\
3\; 3\; 3\; 3 \\ 
4\; 4\; 4\; 4
\end{array} \right] $} 
 & $\left( \begin{array}{cccc}
a_1 & a_2 & a_3 & a_4 \\
a_5 & a_6 & a_7 & a_8 \\
a_9 & a_{10} & a_{11} & a_{12}\\
a_1 + a_5 + a_9 & a_2 + a_6 + a_{10} & a_3 + a_7 + a_{11}  &a_4 + a_8 + a_{12}
\end{array} \right) $\\ 
\hline
{\small $\left[ \begin{array}{c}
1\; 1\; 1\; 1  \\
2\; 2\; 2\; 3 \\
3\; 3\; 3\; 2 \\ 
4\; 4\; 4\; 4
\end{array} \right] $ }
 & Zero Matrix \\
\hline
{\small $\left[ \begin{array}{c}
1\; 1\; 1\; 2  \\
2\; 2\; 2\; 3 \\
3\; 3\; 3\; 1 \\ 
4\; 4\; 4\; 4
\end{array} \right] $ }
 & $\left( \begin{array}{cccc}
a_1 & a_1+a_2 & a_2 & 0 \\
a_2 & a_1 & a_1+a_2 & 0 \\
a_1+a_2 & a_2 & a_1& 0\\
0 & 0 & 0 &0
\end{array} \right) $\\ 
\hline
{\small $\left[ \begin{array}{c}
1\; 1\; 1\; 1  \\
2\; 2\; 4\; 3 \\
3\; 4\; 3\; 2 \\ 
4\; 3\; 2\; 4
\end{array} \right] $ }
 & Zero Matrix\\ 
\hline
{\small $\left[ \begin{array}{c}
1\; 1\; 2\; 2  \\
2\; 2\; 1\; 1 \\
3\; 3\; 3\; 3 \\ 
4\; 4\; 4\; 4
\end{array} \right] $ }
 & $\left( \begin{array}{cccc}
a_1 & a_1 & a_2 & a_3 \\
a_1 & a_1 & a_2 & a_3 \\
a_4 & a_4 & a_5&a_6\\
a_4 & a_4 & a_5 &a_6
\end{array} \right) $\\ 
\hline
{\small $\left[ \begin{array}{c}
1\; 1\; 2\; 2  \\
2\; 2\; 1\; 1 \\
4\; 4\; 3\; 3 \\ 
3\; 3\; 4\; 4
\end{array} \right] $ }
 & $\left( \begin{array}{cccc}
a_1 & a_1 & a_2 & a_2 \\
a_1 & a_1 & a_2 & a_2 \\
a_3 & a_3 & a_4&a_4\\
a_3 & a_3 & a_4 &a_4
\end{array} \right) $\\ 
\hline
{\small $\left[ \begin{array}{c}
1\; 4\; 2\; 3  \\
3\; 2\; 4\; 1 \\
4\; 1\; 3\; 2 \\ 
2\; 3\; 1\; 4
\end{array} \right] $ }
 & $\left( \begin{array}{cccc}
0 & a_1 & a_2 & a_3 \\
a_2+a_3 & 0 & a_2+a_3+a_4 & a_1+a_2+a_3+a_4 \\
a_1+a_3 & a_1+a_4 & 0 &a_1+a_2+a_4\\
a_1+a_2 & a_4 & a_3+a_4 & 0
\end{array} \right) $\\ 
  \hline
\end{tabular}
\end{center}
\end{table}

\subsection{Derivations of dihedral quandles in positive characteristic}
We provide in Table~\ref{Table5}  the matrix forms of the derivations over $\mathbbm k[X]$, where $X$ is the dihedral quandle of cardinality $n$ with $3 \leq n \leq 6$ for the cases of positive characteristic.

Notice that for $n=5$ and the characteristic of $\mathbbm k$ is equal to $2$ or $3$, we have only trivial derivations, corresponding to zero matrices.

\begin{table}[H]
\caption{Matrices of $D$ for dihedral quandles with small values of $n$ with positive characteristic}
\label{Table5}
\begin{center}
\begin{tabular}{ |c|c|cl|} 
 \hline
$n$ & \textit{char}($\mathbbm k$) & \textit{ Matrix of the derivation D } &\\ 
  \hline
$3$ & $3$ & $\left( \begin{array}{ccc}
0 & a_1+a_2 & a_2\\
a_1 &0 &-a_2 \\ 
-a_1 & -(a_1+a_2) &0
\end{array} \right) $  & \\ 
 \hline
$4$ &  $2$ & $\left( \begin{array}{cccc}
a_4 & a_3 & a_4& a_3\\
a_2 &a_1 & a_2& a_1 \\ 
a_4 & a_3 & a_4& a_3\\
a_2 &a_1 & a_2& a_1 \\
\end{array} \right)$  & \\ 
 \hline
 $5$ &  $5$ & $\left( \begin{array}{ccccc}
0 & 4a_1+3a_2 & 4 a_1 & 4a_1+4a_2& 4a_1+a_2\\
4 a_2& 0 &3a_1 &a_1+a_2& 2a_1+3a_2 \\ 
2 a_2& a_1+2a_2 & 0 & 2a_1+2a_2& 3a_1+2a_2\\
3 a_2& 3a_1+a_2 & 2a_1 & 0& a_1+4a_2\\
a_2& 2a_1+4a_2 & a_1 & 3 a_1+3a_2&0 
\end{array} \right)$ & \\
\hline
$6$ &  $2$ & $\left( \begin{array}{cccccc}
a_1 & a_1 & a_2& a_2 & a_2 & a_1\\
a_2 & a_2 & a_2 & a_1 & a_1& a_1 \\ 
a_2 & a_1 & a_1 & a_1 & a_2& a_2 \\
a_1 & a_1 & a_2 & a_2 & a_2& a_1 \\
a_2 & a_2 & a_2 & a_1 & a_1& a_1 \\
a_2 & a_1 & a_1 & a_1 & a_2& a_2 
\end{array} \right)$  & \\
 \hline 
$6$ &  $3$ & \tiny{$\left( \begin{array}{cccccc}
2a_1 & a_1 & 2a_2+a_3+a_4& a_2 & 2a_1+a_2+a_3+a_4 & a_1\\
2a_2+a_3+2a_4 & 2a_2 & 2a_2+a_3+2a_4 & a_1+2a_2 & a_1& 2 a_1+2a_2+a_3 \\ 
a_1+2a_2+2a_3+a_4  & a_1  & 2a_1 & a_1 & a_1+a_2+a_4 & a_2 \\
a_1 & a_1+2a_3  & 2a_2+a_3+2a_4  & 2a_2 & 2 a_2+a_3+2a_4& 2a_1+a_2+2a_3  \\
2a_1+2a_3+a_4  & a_2 & a_4 & a_1 & 2a_1& a_1 \\
2a_2+a_3+2a_4  & a_3 & a_1 & a_2 & 2a_2+a_3+2a_4 & 2a_2 
\end{array} \right)$}  &  \\*[8mm] 
 \hline
\end{tabular}
\end{center}
\end{table}

\subsection{Derivations of conjugation quandle over symmetric group \texorpdfstring{$S_3$}{}}

	Consider the symmetric group $X=S_3 =<x,y; \; x^2=1=y^3,\; xyx=y^{-1}>$ as a quandle with conjugation $u \rt v=v^{-1}uv$.  This quandle decomposes as a disjoint union of orbits $ \{1\}\sqcup \{y, y^2\} \sqcup \{x, yx,y^2x\} $.  Let $\mathbb{Z}[X]$ be its 6-dimensional quandle algebra over $\mathbb{Z}$.  We will denote the basis by $e_1, \ldots, e_6$ corresponding respectively to the quandle elements in the following order $1, y, y^2, x, yx, y^2x$.  

	With the introduced previously notation, using basis vectors, the entry $(i,j)$ in the multiplication table of the algebra is $e_i \cdot e_j:=e_{i \rt j}$, for example $e_1 \cdot e_j=e_1\;$ and $e_j \cdot e_1=e_j, \forall \; 1 \leq j \leq 6$:
	$$\begin{array}{|c|c|c|c|c|c|c|} 
	\hline
	 \ &  e_1 & e_2 & e_3 & e_4 & e_5 & e_6 \\ \hline
	\ e_1 &  e_1 & e_1 & e_1 & e_1  & e_1 & e_1\\ \hline
	\ e_2 &  e_2 & e_2 & e_2 & e_3  & e_3 & e_3\\ \hline
	\ e_3 &  e_3 & e_3 & e_3 & e_2  & e_2 & e_2\\ \hline
	\ e_4 &  e_4 & e_5 & e_6 & e_4  & e_6 & e_5\\ \hline
	\ e_5 &  e_5 & e_6 & e_4 & e_6  & e_5 & e_4\\ \hline
	\ e_6 &  e_6 & e_4 & e_5 & e_5  & e_4 & e_6\\ \hline
	\end{array}
	$$
	The derivations are described as follows:
	\begin{enumerate}[label=\rm \arabic*)]
	    \item If the characteristic is 0 or 2, all the derivations are trivial.
	    \item If the characteristic is 3, then the derivations are described with respect to  the given basis by the matrix
	    
	    \begin{small}\noindent 
	    $\left( \begin{array}{cccccc}
-a_1 & a_1 & a_1&   a_1+a_2+a_3+a_4 &             a_1+a_2+a_3+a_4&      a_1+a_2+a_3+a_4\\
-a_1 & -a_4&-a_1+a_4 &  a_1+a_2+a_3+a_4 &         a_1+a_2+a_3+a_4&        a_1+a_2+a_3+a_4 \\ 
-a_1& -a_1+a_4& -a_4 &  a_1+a_2+a_3+a_4 &       a_1+a_2+a_3+a_4&            a_1+a_2+a_3+a_4  \\
a_1 & a_4  & a_4  &     -a_1-a_2-a_3-a_4 &          a_1+a_3+a_4&        -a_1+a_2-a_3  \\
a_1  & a_4  & a_4 & a_1+a_2+a_4 &              -a_1-a_2-a_3-a_4&          -a_1+a_2-a_3 \\
a_1 & a_4  & a_4 &       a_3 &                       a_2 &             -a_1-a_2-a_3-a_4 
\end{array} \right)$
\end{small}	   
	\end{enumerate}
	
\section{Lie Transformation Algebra and Inner Derivations}\label{LieTransformation}

Recall that if $A$ is a non-associative algebra, a subspace $I\leq A$ is an ideal of $A$ if it is two-sided ideal, that is for every $a\in A$ and $x\in I$, we have $xa\in I$ (left ideal) and $ax\in I$ (right ideal). A nontrivial non-associative algebra is said to be {\it simple}, if its only ideals are ${0}$ and $A$. Recall further, that a {\it semisimple} algebra is a direct sum of simple algebras: $A = A_1\oplus \ldots \oplus A_n$. In this section, we do not make any assumptions on the characteristic of the ground field, as the results are valid on any characteristic.

The {\it center} $Z$ of a non-associative algebra $A$, is the subspace of all the elements $x\in A$ that commute and associate with everything. In other words, 
for all $a,b\in A$, 
$$
xa = ax,\ \ a(xb) = (ax)b.$$
Observe that if $A$ is associative, the second set of equalities automatically holds, and therefore the center coincides with the usual definition.

When $A$ is an associative algebra, it is easy to see that the map $[\bullet,x] : A \longrightarrow A$ defines a derivation of $A$.
This result relies on the Jacobi identity, as a direct inspection proves. In non-associative algebras, the Jacobi identity is replaced by a more general relation. In fact, let $I,J,K$ be ideals of $A$. Then,
$$
[[I,J],K] \subset [[J,K],I] + [[K,I],J],
$$
where $[I,J]$ means the subspace of elements $[x,y]$ with $x\in I$ and $y\in J$. We can modify the notion of inner derivation in the following way \cite{Scha}. Let $T$ be a set of linear maps $A \longrightarrow A$. Consider iteratively
$$
T_i = [T_1,T_{i-1}],
$$
where we set $T_1 = T$. Then $\mathcal{T}(A) := T_1 + \ldots + T_K + \ldots $ is the smallest Lie algebra containing $T$ (see \cite{Scha} for a proof). In the specific case in which $T$ consists of right and left multiplication maps by elements of $A$, $\mathcal{T}(A)$ is called {\it Lie transformation algebra} of $A$.

\begin{definition}\label{def:nonassinner}
	{\rm 
		Let $A$ be a non-associative algebra and let $\mathcal{T}(A)$ be its Lie transformation algebra. Then a derivation of $A$ which is also an element of $\mathcal{T}(A)$ is called an {\it inner derivation} of $A$.
	}
	\end{definition}
\begin{remark}
	{\rm 
It is known that when $A$ is associative, then Definition~\ref{def:nonassinner} and the usual definition of inner derivation coincide. 
}
\end{remark}

\begin{remark}
    {\rm 
    The notions of simple and semi-simple algebra are defined in the context of non-associative algebras in a similar manner to the associative case. Namely, a simple (non-associative) algebra is an algebra which does not contain nontrivial subalgebras, while an algebra is semi-simple if it is the direct sum of simple algebras. It is a result of Schafer, see \cite{Scha}, that for semi-simple algebras, derivations are all inner if and only if all the derivations of the simple algebras are inner. This paradigm can be applied to the case of quandle algebras, with inner derivations as in Definition~\ref{def:nonassinner}.
    }
\end{remark}

Let us denote by $\mathcal L(X)$ the vector space generated by left multiplications of elements in $\mathbbm k[X]$. Let $\psi = \sum_i a_i \cdot e_{x_i}$ be in $\mathbbm k[X]$, denote left multiplication map by an element of $\psi$  by the symbol $L_{\psi}$. Using linearity of the product in $\mathbbm k[X]$, $L_\psi$ can be written as a linear combination of  $\sum_i a_i L_{e_{x_i}}$. For simplicity we write $L_{x_i}$ for $L_{e_{x_i}}$. Similarly, we denote by  $\mathcal R(X)$ the linear space generated by right multiplications of elements in $\mathbbm k[X]$, where we adopt a similar notation to indicate right multiplication operator by a basis vector $e_{x_i}$. By definition, therefore, the Lie transformation algebra $\mathcal T(\mathbbm k[X])$ is given by the smallest Lie algebra containing $\mathcal L(X) + \mathcal R(X)$, and an inner derivation of a quandle algebra is an element of $\mathcal T(\mathbbm k[X])$ which satisfies the Leibniz rule.

Self-distributivity of basis elements in a quandle/rack algebra induces a commutation relation of maps of type $L_y$ and $R_x$ that forces the Lie transformation algebra of $\mathbbm k[X]$ to be determined by products of left and right multiplications in a definite order. Specifically, we have the following result.

\begin{proposition}\label{lem:Lietransf}
	Let $X$ be a quandle and let $A := \mathbf k [X]$ be its quandle algebra. Then the Lie transformation algebra $\mathcal{T}(A)$, is contained in:
	$$
	\mathcal {L R}(A) := \sum_{n,m\in \mathbbm{N}_0} L^m(A) \cdot R^n(A),
	$$
	where $L^m(A)$ denotes the vector space generated by products of left multiplications of $A$, $m$ times, and $R^n(A)$ denotes the space generated by products of right multiplications, $n$ times.  
\end{proposition}
\begin{proof}
	As observed before, that if $x\in A$, it is by definition of the form, $\sum_i a_i x_i$ and therefore the right multiplication map $R_x : A \longrightarrow A$ can be written as a linear combination: $R_x = \sum_i a_iR_{x_i}$. Now, let $T_1 := R(A) + L(A)$ and consider the commutator $[T_1,T_1]$. Let $x = \sum_i a_i x_i$ and $y = \sum_i y_j$, from:
	$$
	R_x L_y = \sum_{i,j} R_{x_i}L_{y_j} = \sum_{i,j} L_{L_{y_j}(x_i)}R_{x_i},
	$$
	where the second equality is obtained by means of the self-distributivity in the quandle $X$, it follows that $R(A)L(A) \subseteq L(A)R(A)$, and we therefore obtain that $T_2 = [T_1,T_1] \subseteq L(A)L(A) + L(A)R(A) + R(A)R(A)$. Inductively, assuming $T_k = \sum_{n,m=0}^{k} L^m(A) \cdot R^n(A)$, we have that $$T_{k+1} := [T_1,T_k] \subseteq \sum_{n,m=0}^{k} [L(A), L^m(A) \cdot R^n(A)] + \sum_{n,m=0}^{k} [R(A),L^m(A) \cdot R^n(A)].$$
	But since $R(A)\cdot L(A) \subseteq L(A)\cdot R(A)$, it follows that $[L(A), L^m(A) \cdot R^n(A)] \subseteq L^{m+1}\cdot R^n(A)$ and $[R(A),L^m(A) \cdot R^n(A)]\subseteq L^m(A) \cdot R^{n+1}(A)$. This concludes the proof.
\end{proof}
\begin{remark}
    {\rm 
    It is clear that the sum in Proposition~\ref{lem:Lietransf} is finite, since the space of linear maps is finite dimensional when $X$ is a finite quandle. The proof is indeed valid also for infinite quandles, but it is not clear at this point whether the Lie transformation algebra is finite dimensional.
    }
\end{remark}
\begin{remark}
	{\rm 
	In fact, an argument similar to that given in the proof of 
	Proposition~\ref{lem:Lietransf}, shows that $\mathcal {L R}(A)$ is a Lie algebra. This fact provides an alternative proof to the inclusion given above, since the Lie transformation algebra of $A$ is by definition the smallest Lie algebra containing the right and left multiplication maps. 
}
	\end{remark}

We now consider the case of Alexander quandle algebras, where $X$ is a $\Z[T,T^{-1}]$-module and the quandle structure is determined by the operation $x\rt y = Tx + (1-T)y$. For simplicity of notation, we set $T = \alpha$ and $1-T = \beta$. 

\begin{proposition}\label{lem:alex}
	Let $X$ be an Alexander quandle and $A:= \mathbbm k[X]$ be its quandle algebra. Then any element of the Lie transformation algebra $\mathcal T(A)$ can be written as a sum of type
	$$
	S = L_{f_{00}} + L_{f_{10}}L_0 + L_{f_{01}}R_0 + \ldots + L_{f_{nm}}L_0^nR_0^m + R_{g_{00}} + R_{g_{10}}L_0 + R_{g_{01}}R_0 + \ldots + R_{g_{nm}}L_0^nR_0^m,
	$$ 
	where $f_{ij}$ and $g_{ij}$ are basis elements of $A$.
\end{proposition}
\begin{proof}
	First we consider the first commutator $[\mathcal L(X) +\mathcal R(X), \mathcal L(X) +\mathcal R(X)]$ defining the Lie transformation algebra of $A$. For $e_x$, $e_y$ and $e_z$ basis elements of $A$, we have that 
	\begin{eqnarray*}
	[L_x,L_y] (e_z) &=& e_{\alpha x+\alpha\beta y+\beta^2 z} - e_{\alpha y+\alpha\beta x+\beta^2 z}\\ 
	&=& L_{x+\beta y}e_{\beta z} - L_{y+\beta x} e_{\beta z}\\ &=& L_{x+\beta y}L_0 e_z - L_{y+\beta x}L_0e_z. 
	\end{eqnarray*}
Similarly, we obtain 
$$
[L_x,R_y](e_z) = L_{x+\frac{\beta^2}{\alpha}y}R_0 e_z - L_{\alpha x+\frac{\beta}{\alpha}y} R_0 e_z
$$
and 
$$
[R_x,R_y] (e_z) = R_{\alpha y + x}R_0 e_z - R_{\alpha x+ y}R_0 e_z. 
$$
We now proceed by induction on the commutators $[\mathcal L(X) +\mathcal R(X), \mathcal L(X) +\mathcal R(X)]^n := [\mathcal L(X) +\mathcal R(X), [\mathcal L(X) +\mathcal R(X), \mathcal L(X) +\mathcal R(X)]^{n-1}]$. By direct computation, we have
$$
	[L_xL_0^nR_0^m, L_y] = L_{x+\alpha^{m+1}\beta^{n+1}y} L_0^{n+1}R_0^m(e_z) - L_{y+\beta x} L_0^{n+1}R_0^m(e_z),
	$$
	and also $[L_xL_0^nR_0^m, R_y]$ and $[R_xL_0^nR_0^m, L_y]$ can be written as a difference of elements of type $L_zL_0^pR_0^q$, where either $p = n+1$ and $q = m$ or $p = n$ and $q = m+1$. Similarly, $[R_xL_0^nR_0^m, R_y]$ is a difference of maps $R_zL_0^nR_0^{m+1}$ for some $z$. From the inductive hypothesis it follows that 
	\begin{eqnarray*}
	\lefteqn{ [\mathcal L(X) +\mathcal R(X),  [\mathcal L(X) +\mathcal R(X), \mathcal L(X) +\mathcal R(X)]^{n-1}]}\\ & \subset& \sum_{i = 0}^{n+1} [\mathcal L(X)\cdot L_0^iR_0^{n+1-i} + \mathcal R(X)\cdot L_0^iR_0^{n+1-i} ].
\end{eqnarray*}
	Therefore an element of the Lie transformation algebra of $A$ is as in the statement. 
\end{proof}

Observe that when $X$ is finite, then $\alpha$ and $\beta$ have finite order, denoted ${\rm ord}(\alpha)$ and ${\rm ord}(\beta)$, respectively. Therefore the sum in Proposition~\ref{lem:alex} can be taken with $n\leq {\rm ord}(\alpha)$ and $m\leq {\rm ord}(\beta)$. Moreover, from the proof it is also clear that $[L_x,R_x] = 0$, which is not an obvious fact in general.

\begin{example}
{\rm 
In this example we give explicit computations of the Lie transformation algebra $\mathcal T(A)$, where $A:= \mathbbm k[X]$ for the three quandles of order $3$.

\begin{enumerate}[label=\rm \arabic*)]

\item
For the trivial quandle of order $3$, the Lie transformation algebra $\mathcal T(A)$ is $3$-dimensional with a basis $(L_1, L_2, L_3)$ and brackets defined as 
$[L_i,L_j]=L_{i}-L_{j} 
$, with $i,j\in \{1,2,3\}$.  In fact for any trivial quandle of order $n$, the Lie transformation algebra is $n$-dimensional with the same structure. 
   
  \item
  Let us consider the quandle $X$ on $3$ elements $\{1,2,3\}$ determined by right multiplications $R_1 = R_2 = \mbox{id}$, and $R_3 = (1 \;2)$. We indicate by the same symbols the right multiplications on the quandle algebra $A := \mathbbm k[X]$ over a field $\mathbbm k$ of characteristic zero. The elements of $A$, as before, are written $e_k$ for $k=1,2,3$. The corresponding left multiplications are as follows, $L_i e_j = e_i$, where $i,j = 1,2$, $L_1 e_3 = e_2$ and $L_2 e_3 = e_1$, and $L_3 = P_3$, where $P_i$ is the map that projects every vector to $e_i$. By direct computation we see that $[L_1,L_2] = P_1 - P_2$, $[L_1,L_3] = P_2 - P_3$ and $[L_2,L_3] = P_1-P_3$. In particular, one has $[L_1,L_2] + [L_1,L_3] = [L_2,L_3]$. Moreover, we have $[L_1,R_3] = - [L_2,R_3] = L_1-L_2$, $[L_3,R_3] = 0$ and, of course, $[L_i,R_j] = 0$ for all $i = 1,2,3$ and $j = 1,2$. This computations determine the commutator space $[T_1,T_1]$, where $T_1$ is the linear space generated by right and left multiplications. The higher commutators $[T_i,T_1]$ do not generated new elements in the Lie transformation algebra of $A$, since the commutators of left and right multiplications with the projectors $P_i$ can be written as combinations of the projectors $P_i$. We conclude that $\mathcal T(A) = \langle \mbox{id}, L_1, L_2, R_3, P_1, P_2, P_3\rangle$. 
  
  \item
  For the dihedral quandle on 3 elements $\{1,2,3\}$ given by $R_1=L_1=( 2\;3), R_2=L_2=(1\;3)$ and $R_3=L_3=(1\;2)$.  The algebra $A := \mathbbm k[X]$ has basis $e_i$ for $i=1,2,3$.  In particular, the Lie transformation algebra is generated by left multiplications. To simplify the computation of commutators, we will change the basis $(e_1,e_2,e_3)$ to $(u,v,w)$, where $u=e_1+e_2+e_3$, $v=e_2-e_1$ and $w=e_3-e_1$. Since $u$ is fixed by all $L_k$, we thus have $[L_i,L_j](u)= 0$.  Now the actions of $L_1, L_2$ and $L_3$ on $(v,w)$ are given respectively by the matrices $\begin{pmatrix} 0 &1\\1&0 \end{pmatrix}$, $\begin{pmatrix} 1 &0\\-1&-1 \end{pmatrix}$ and $\begin{pmatrix} -1 &-1\\0&1 \end{pmatrix}$.  Now direct computations give that $[L_1,L_2]$ acts on $(v,w)$ as the matrix $\begin{pmatrix} -1 &-2\\2&1 \end{pmatrix}$ and thus $[L_1,L_2]$ is independent of $L_1, L_2, L_3$.  The other commutators are given by $[L_1,L_3]=-[L_1,L_2]$, $[L_2,L_3]=[L_1,L_2]$.  This computation gives that the commutator space $T_2=[T_1,T_1]$.  Now we have $[L_1,[L_1,L_2]]=2L_1+4L_2$, $[L_2,[L_1,L_2]]=2L_3-2L_1$, and $[L_3,[L_1,L_2]]=2L_1-2L_2$.  Thus the higher commutators $[T_i,T_1]$ do not generate any new elements in the Lie transformation algebra of $A$.  We then conclude that $\mathcal T(A) = \langle \mbox{id}, L_1, L_2, L_3, [L_1,L_2]\rangle$.    
\end{enumerate}
}   
\end{example}


\section*{Acknowledgement} The authors would like to thank Professor Ivan Shestakov for fruitful discussions, and Mahender Singh for useful suggestions. 
Mohamed Elhamdadi was partially supported by Simons Foundation collaboration grant 712462. Emanuele Zappala was supported by the Estonian Research Council through the Mobilitas Pluss scheme, grant MOBJD679. Sergei Silvestrov is grateful for support from The Royal Swedish Academy of Sciences Foundations.

\end{document}